\documentclass[a4paper, 14pt]{extreport}
\usepackage[T2A]{fontenc}
\usepackage[utf8]{inputenc}
\usepackage[english]{babel}
\usepackage{amssymb,amsmath,amsthm,amsfonts}

\usepackage{indentfirst}

\usepackage[pdftex]{graphicx}
\graphicspath{{pictures/}}
\DeclareGraphicsExtensions{.pdf,.png,.jpg}

\usepackage{geometry}
\newcommand{\bv}{{\boldsymbol v}}
\newcommand{\bw}{{\boldsymbol w}}

\newtheorem{theorem}{Theorem}
\newtheorem{lemma}[theorem]{Lemma}
\newtheorem{definition}{Definition}
\newtheorem{remark}{Remark}
\newtheorem{corollary}{Corollary}
\newtheorem{conj}{Conjecture}

\begin{document}

\begin{center}
  \textbf{Karapetyants M.A.}
\end{center}

\begin{center}
    \textbf{Subdivision schemes on a dyadic half-line}
\end{center}

\begin{center}
  \textbf{Abstract}
\end{center}

In this paper subdivision schemes, which are used for functions approximation and curves generation, are considered. In classical case, for the functions defined on the real line, the theory of subdivision schemes is widely known due to multiple applications in constructive approximation theory, signal processing as well as for generating fractal curves and surfaces. Subdivision schemes on a dyadic half-line – positive half-line, equipped with the standard Lebesgue measure and the digitwise binary addition operation, where the Walsh functions play the role of exponents, are defined and studied. Necessary and sufficient convergence conditions of the subdivision schemes in terms of spectral properties of matrices and in terms of the smoothness of the solution of the corresponding refinement equation are proved. The problem of the convergence of subdivision schemes with non-negative coefficients is also investigated. Explicit convergence criterion of the subdivision schemes with four coefficients is obtained. As an auxiliary result fractal curves on a dyadic half-line are defined and the formula of their smoothness is proved. The paper contains various illustrations and numerical results.

Bibliography: 26 items.


\subsection{Introduction}

\subsubsection{Dyadic half-line}

We begin with the definition of the dyadic half-line \cite{G, GES}. We consider sequences

$$
x = \{ \ldots, x_{j-1}, x_j, x_{j+1}, \ldots \},
$$

where $ x_j \in \{0, 1\} $ and $ x_j = 0 $, $ j > k $, for sufficiently large $ k > 0 $. Each element $x$ corresponds to a convergent series $\sum_{n\in{{\mathbb Z}}}x_n 2^n$, which is the binary decomposition of a number in $ \mathbb R_+ = [0; \infty)$. Define algebraic summation operation as follows: sum of two sequences $ x $ and $ y $ is sequence $ z $, such that

$$
z \ = \  x \oplus y \ = \ \{ x_i \oplus y_i \}
$$

for each $i$;

$$
    x_i \oplus y_i = \left\{ \begin{array}{cc}
    0, & x_i + y_i \in \{0, 2\}, \\
    1, & x_i + y_i = 1.
    \end{array}\right.
$$

It follows, that $x \oplus x \ = \ 0 \ $ for each $x \in {\mathbb R}_+$, which implicates the coincidence of $\oplus$ operation and inverse operation $\ominus$. For instance,

$$
    3 = 11_{2}, \ 6 = 110_{2}, \ 3 \oplus 6 = 5
$$

The continuity of a dyadic function on ${\mathbb R}_+$ in literature is called $W$-{\em continuity} (after Walsh, it seems). A function $f$ is $W$-continuous at the point $x$, if

$$
  \lim_{h \to 0} |f(x \oplus h) - f(x)| = 0
$$

Basically, $W$-continuous function $f$ is continuous from the right at dyadic rational points of the half-line and continuous in usual sense at all others.
Half-line ${\mathbb R}_+$ equipped with $\oplus$ operation and $W$-continuity we call {\em dyadic half-line}.

Unit shift of the dyadic half-line is represented below: linear function alters as it is shown on the figure 1.

\begin{figure}[h]
\begin{minipage}[h]{0.49\linewidth}
\center{\includegraphics[width=1\linewidth]{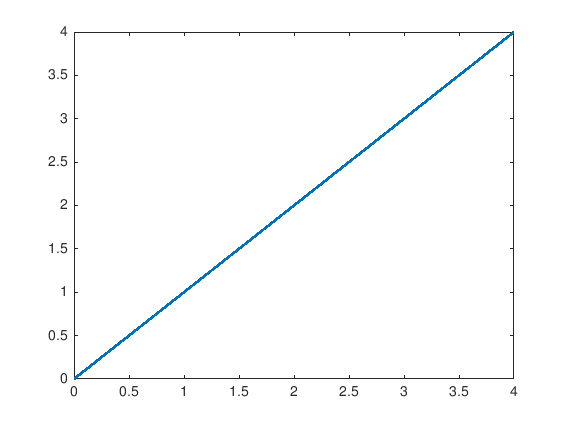} \\ а)f(x)}
\end{minipage}
\hfill
\begin{minipage}[h]{0.49\linewidth}
\center{\includegraphics[width=1\linewidth]{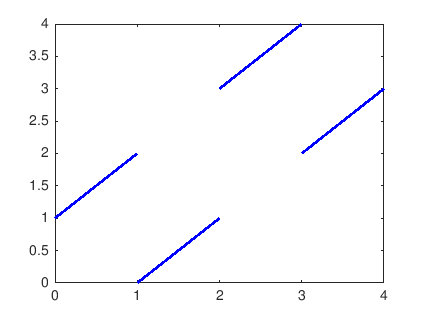} \\ б)f(x $\oplus$ 1)}
\end{minipage}
\caption{Unit shift of the dyadic half-line.}
\label{ris:image0}
\end{figure}

\subsubsection{Subdivision schemes}

The first examples of subdivision schemes (in the classical case, that is, on a straight line with regular addition) appeared in the early fifties in the works of G. de Rham (the famous algorithm of ''cutting corners'' \cite{P_4, DeR}). Then this idea was developed in the works of G. Chaikin \cite{Ch}, G. Deslauriers and S. Dubuc \cite{DD}. The general theory of subdivision algorithms was represented by N. Dyn, A. Levin, C. Micchelli, V. Dahmen and A. Kavarette, K. Conti, P. Osvald, and other researchers. (eg. \cite{CDM, Dyn, DynLG, DynLM} and references in these papers). All these papers are devoted to the approximation of functions, or the construction of fractal curves on the usual (classical) line, or generalizations to functions of several variables. We restrict ourselves to the functions of one variable. Let us recall the basic concepts and facts about the subdivision schemes on the real line.

Consider the space $\ell_\infty$ and a linear operator $ \tilde{S}: \ell_\infty \rightarrow {\ell}_\infty$, defined by a finite set of coefficients called {\em mask}: $ { \{ c_j \} }_{j \in \mathbb J} $. This operator acts according to the following rule: for any sequence $a \in {\ell}_{\infty}$ 

$$
(\tilde{S} a)(\alpha) \ = \ \sum_{\beta \in {{\mathbb Z}}} c_{\alpha - 2\beta} \  a_\beta
$$

One iteration of {\em subdivision scheme} consists in the application of the operator $ \tilde{S} $ to some sequence $a$.
We match sequence $a$ piecewise linear function with nodes at integer points, with the value of the function at the node $k$ equal to $a_k$. So we do with the sequence $ \tilde{S} a $, but with nodes in half-integer points (on $\frac12 \mathbb Z$). We say, that {\em subdivision scheme converges}, if for any bounded sequence there is a continuous function $f = f_a$ such, that

$$
    \lim_{n \to \infty} \|\ f(2^{-n} \cdot) - (\tilde{S}^na)(\cdot) \|_{\infty} = 0
$$

Thus, if the scheme converges, then for each sequence $a$ there is a continuous function~$f_a$.

{\em Dyadic subdivision operator} matches sequence $a \in {\ell}_{\infty}$ sequence $Sa$, defined as follows:

$$
(Sa)(\alpha) \ = \ \sum_{\beta\in{{\mathbb Z}}} c_{\alpha \ominus 2\beta} \  a_\beta
$$

By {\em dyadic subdivision scheme} we mean the sequential application of the operator $S$ to some bounded sequence $a$. The convergence of the dyadic subdivision scheme is determined similarly to the classical one, but the limit function will be defined on $ \mathbb R_+ $, and will be $W$-continuous.

\begin{definition}\label{Convergence}
    A dyadic subdivision scheme converges if, for any bounded sequence, there exists a $W$-continuous function $f$ such that
    $$
        \lim_{n \to \infty} \|\ f(2^{-n} \cdot) - (S^n a)(\cdot) \|_{\infty} = 0
    $$
\end{definition}

This paper is devoted to the study of various properties of dyadic subdivision operators and schemes, as well as the associated dyadic fractals. In Section 2 the necessary and sufficient conditions for the convergence of the dyadic subdivision schemes are formulated. In Section 3 the conditions under which two affine operators generate a dyadic fractal curve are investigated. This result is used in the construction of limit functions of subdivision schemes. Section 4 examines some spectral properties of subdivision schemes. Sections 5 and 6 are devoted to dyadic subdivision schemes with positive and non-negative coefficients, respectively. Section 7 discusses in detail the simplest special case of the dyadic scheme defined by four coefficients, and its explicit convergence conditions are obtained. Section 8 consists entirely of examples and illustrations of various dyadic subdivision schemes.

\subsection{Necessary and sufficient convergence conditions}

In order to apply a subdivision scheme to interpolation problems, it is necessary for this scheme to converge. First of all, it should be established under what conditions the subdivision scheme is guaranteed to converge for any initial sequence. To do this, we need several auxiliary assertions.
Define $\delta$-sequence as follows: $\delta(\cdot ) = (1, 0,0, \ldots)$.

\begin{lemma}\label{Delta}
    The subdivision scheme converges if and only if it converges on $\delta$-sequence.
\end{lemma}

\begin{proof}
    If the subdivision scheme converges, then, obviously, it converges on $\delta$-sequence. Now, note that any bounded sequence can be represented as a linear combination of shifts of $\delta$-sequence:
    \begin{equation}\label{linear combination of delta}
    a \ = \ \sum_{k \in{{\mathbb Z}}} \delta(\cdot \ominus k) a_k
    \end{equation}
    Therefore, if the subdivision scheme converges on $\delta$-sequence, then it converges on any bounded sequence due to the linearity of the subdivision operator.
\end{proof}

This fact means that it is sufficient to verify the convergence on $\delta$-sequence, and if it is possible to establish the fact of convergence on $ \delta $, then it is transferred to all other sequences in view of the lemma.

\begin{lemma}
    Let $\varphi$ be limit function of the subdivision scheme on $\delta$-sequence. If there is a limit function $f_a$ on any sequence $a \in \ell_\infty$, then it is given as follows:
    \begin{equation}\label{representation of a random function}
    f \ = \ \sum_{k \in{{\mathbb Z}}}\, {a}_k\, \varphi(x \ominus k)
    \end{equation}
\end{lemma}

\begin{proof}
    As mentioned above, any limited sequence can be represented as (\ref{linear combination of delta}). We apply $k$ times to the both parts of (\ref{linear combination of delta}) operator $S$ and tend $k$ to infinity. For each term on the right side of the equality, there is a limit function, therefore, the limit function for the left side of the equality also exists and is given by (\ref{representation of a random function}).
\end{proof}

So, let $\varphi$ be the limit function of the subdivision scheme on the delta sequence (provided the scheme converges). Is it possible to find this function without using a subdivision scheme? The answer is affirmative: this function is the solution of a certain difference equation with the compression of the argument.

\begin{lemma}
     Let $\varphi$ be a limit function of a subdivision scheme. Then $\varphi$ is the solution of the following equation

\begin{equation}\label{refinement equation}
    \varphi(x)\  = \ \sum_{k\in{{\mathbb Z}}} c_k \varphi(2x \ominus k)
\end{equation}

\end{lemma}

\begin{remark}
Equation (\ref{refinement equation}) is called refinement regarding $\varphi$.
\end{remark}

\begin{proof}
    Let the subdivision scheme with the coefficients  $ \{c_k\}_{k = 0}^N $ converge and $\varphi$ be its limit function on $\delta$-sequence. Since after one iteration $S$ on $\delta$ we get a sequence consisting entirely of coefficients that define subdivision operator, then
    $$
    S^n(\delta) = \sum_{k\in{{\mathbb Z}}} c_k S^{n-1}(\delta(\cdot \ominus k)).
    $$
    On each step $n$ we match $\ S^n(\delta)$ with the sequence on $R_+$ with interval $2^{-n}$, beginning from zero, and connect the obtained points piecewise linearly. The functions thus obtained $f_n$ will converge to $\varphi$ due to the convergence of the scheme:
    $$
    \varphi = \lim_{n\to\infty}f_n,
    $$
    and, similarly
    \begin{equation}\label{Representation of f_n}
        f_n(\cdot) = \sum_{k\in{{\mathbb Z}}} c_k f_{n-1}(\cdot \ominus k).
    \end{equation}
    It only remains to note that multiplying the initial sequence by a constant is equivalent to multiplying the limit function by a constant, while shifting the argument of the initial sequence shifts the argument of the limiting function. Passing to the limit as $n \to \infty$ in equality (\ref{Representation of f_n}) and given that we have done one iteration "manually", we obtain
    $$
    \varphi(x) = \sum_{k\in{{\mathbb Z}}} c_k \varphi(2x \ominus k).
    $$
\end{proof}

Not any subdivision scheme converges. Our immediate goal is to obtain convergence conditions. We start with the necessary conditions. They repeat exactly the same conditions for the classical subdivision schemes on the line.

Function $f$ is called uniformly continuous on $\mathbb R_+$, if for each $ \varepsilon > 0$ there exists $ \delta > 0 $ such, that $ \| f(y) - f(x) \| < \varepsilon $, as soon as $ \| y \ominus x \| < \delta, \ \  x, y \in \mathbb R_+ $.

\begin{theorem}\label{Necessary convergence condition}
    Given a subdivision scheme with mask $ \{c_k\}_{k = 0}^N $. If the scheme converges, then
\begin{equation}\label{eq.nec}
    \sum_{k \in {{\mathbb N}}} c_{2k} = \sum_{k \in {{\mathbb N}}} c_{2k+1} = 1.
\end{equation}

\end{theorem}

\begin{proof}
    Let the initial sequence be given:
    $$
    a \ = \ \{\ldots, a_{-1}, a_0, a_1, \ldots\}.
    $$
    It can be associated with the function $ f_0(x) $ defined on a uniform lattice on $ R_+ $ with a step one. Similarly the sequence $Sa$ could be associated with the function $f_1(x)$, but with a different step: $\frac12$, etc. Consider the sum
    $$
    \sum_{k \in {{\mathbb N}}} c_{2k} f_n(x_{2k}),
    $$
    where $f_n(x_{2k})$ is value of function $f_n(x)$, associated with coefficient $c_{2k}$. Since the scheme converges to some function f(x), then for any $ \varepsilon > 0 $ there is a number $N$ such, that for each $n > N$ the inequality $ |f_n(x)-f(x)| < \varepsilon $ holds.
    For each n $\geq$ N on the iteration n intervals between function values $f_n(x)$ are small and equal to $2^{-n}$, consequently, the values of function $f_n(x)$ do not differ much, because for $f_n(x)$ there is a uniform continuity on each of the segments whose ends are binary rational points. Let for $ {\varepsilon}_n = 1/n $ be a number $ n_0 > N $ such, that for each $ p \in \mathbb N $ the inequality $ |f_{n_0+p}(x)-f(x)| < {\varepsilon}_n $ holds.
    Consider one of these segments and denote ${\hat f}_n$ as a value of function $f_n(x)$ at the left end of this segment. We suppose that on $n+1$ step the values of the function $f_{n+1}(x)$ differ in the corresponding points from ${\hat f}_n$ no more than $2{\varepsilon}_n$ (otherwise, we will select a new $ n_0 $ to satisfy this condition). On the other hand,
    $$
    \sum_{k \in {{\mathbb N}}} c_{2k} ( {\hat f}_n - 2{\varepsilon}_n ) \ \leq \ f_{n+1}(x) \ = \ \sum_{k \in {{\mathbb N}}} c_{2k} f_n(x_{2k}) \ \leq \ \sum_{k \in {{\mathbb N}}} c_{2k} ( {\hat f}_n + 2{\varepsilon}_n ),
    $$
    where ${\varepsilon}_n \to 0$ as $n \to \infty$. But $|f_{n+1}(x) - {\hat f}_n| < 2{\varepsilon}_n$, from which and the last inequality it follows that
    $$
    \sum_{k \in {{\mathbb N}}} c_{2k} = 1.
    $$
    Similarly, it is proved, that $ \sum_{k \in {{\mathbb N}}} c_{2k+1} = 1 $.
\end{proof}

Theorem 4 is a necessary condition for the convergence of a subdivision scheme in terms of its mask. What is sufficient for the convergence of the subdivision scheme? To answer this question we need to state a few auxiliary definitions.

In future, unless otherwise specified, we will omit the symbol $ W $ in the designation of $W$-continuity, i.e. by continuity we mean $W$-continuity. By $C(\mathbb R_+)$ we define a space of $W$-continuous functions.

Let $f$ be a compactly supported continuous function.

\begin{definition}\label{Stability of a function}
    We say, that function $f$ possesses linearly independent integer shifts if for any sequence $a$ 

    \begin{equation}\label{eq.stab}
        \sum_{k\in{{\mathbb Z}}} a_k f(x \ominus k) \ \not\equiv \ 0
    \end{equation}

    The function $f$ is called {\em stable} if~(\ref{eq.stab}) is satisfied for any $a \in {\ell}_\infty$.
\end{definition}

We define now the linear operator $T:C(\mathbb R_+) \rightarrow C(\mathbb R_+)$, acting according to the rule

$$
T(f(x)) \ = \ \sum_{k=0}^N p_k f(2x \ominus k),
$$

where $ \{ p_k \}_{k = 0}^N $ is a finite set of coefficients.
In literature (in the classical case) this operator is called the {\em transition operator}. We now establish a connection between the operators $S$ and $T$.

\begin{lemma}\label{Transition operator}
    Let  $f(x)$ be continuous and compactly supported. There is an equality:
    $$
    (T^nf)(x) = \sum_{k \in {{\mathbb Z}}} (S^n \delta)(k) f(2^nx \ominus k)
    $$
\end{lemma}

\begin{proof}
    We prove this formula by induction. For $n \ = \ 1$ we have
    $$
    Tf(x) = \sum_{k=0}^{N} p_k f(2x \ominus k).
    $$
    On the other side,
    $$
    \sum_{k=0}^N p_k f(2x \ominus k) \ = \ \sum_{k=0}^N (S(\delta))(k) f(2x \ominus k).
    $$
    Let the formula be true for $n-1$:
    $$
    (T^{n-1}f)(x)\ = \ \sum_{k=0}^{N} (S^{n-1}(\delta))(k) f(2^{n-1}x \ominus k).
    $$
    Consider $T(T^{n-1}f)(x) = T^{n}(x)$:
    $$
    T(\sum_{k=0}^N p_k f(2^{n-1}x \ominus k)) \ = \ T(\sum_{k=0}^N (S^{n-1}(\delta))(k) f(2^{n-1}x \ominus k) \ = \
    $$
    $$
    = \ \sum_{j=0}^N \sum_{k=0}^N (S^{n-1}(\delta))(k) p_j f(2^nx \ominus  j \ominus  2k) \ = \ \sum_{j=0}^N \sum_{k=0}^N (S^{n-1}(\delta))(k) p_{j \ominus 2k} f(2^nx \ominus  j) \ = \
    $$
    $$
    = \ \sum_{j=0}^N (S^n(\delta))(j) f(2^nx \ominus  j).
    $$
    Lemma is proved.
\end{proof}

Now we formulate a theorem providing conditions sufficient for the convergence of the scheme. Its proof practically repeats the classical case.

\begin{theorem}\label{Sufficient convergence condition}
    If the refinement equation
    $$
    \varphi(x) \ = \ \sum_{k \in{{\mathbb Z}}} c_k \varphi(2x \ominus k)
    $$
    has a continuous and stable solution $f(x)$, then the corresponding subdivision scheme converges.
\end{theorem}

\begin{proof}
    Due to the lemma \ref{Delta} we only have to prove this theorem for $\delta$-sequence. By lemma \ref{Transition operator}:
    $$
    (T^n f)(x) = \sum_{k \in {{\mathbb Z}}} (S^n(\delta))(k) f(2^n x \ominus k).
    $$
    If the scheme converges, then it converges to a continuous function. $f(x)$ and
    $$
    f(x) = \sum_{k \in {{\mathbb Z}}} (S^n(\delta))(k) f(2^n x \ominus k),
    $$
    because $f(x)$ is an eigenfunction of $T$. Consider 1-periodical function $F(x) = \sum_{k \in {{\mathbb Z}}} f(x \ominus k)$. We show that it is constant. Indeed,
    $$
    \sum_{i \in {{\mathbb Z}}} f(\frac{x}{2} \ominus i) \ = \ \sum_{i \in {{\mathbb Z}}} \sum_{j \in {{\mathbb Z}}} {c}_j  f(x \ominus 2i \ominus j) \ = \
    $$
    $$
    \ = \ \sum_{i \in {{\mathbb Z}}} \sum_{j \in {{\mathbb Z}}} {c}_{j-2i} f(x \ominus j) \ = \ \sum_{i \in {{\mathbb Z}}} f(x \ominus i),
    $$
    since in order for the scheme to converge, it is necessary that the sum of all even (as well as odd) coefficients are equal to one. We obtained, that $F(\frac{x}{2}) = F(x)$ - the function is constant at half-integer points. Similarly, we obtain that it is constant at all binary rational points, and since this set is everywhere dense on ${\mathbb R}_+$, then $F(x)$ is constant on ${\mathbb R}_+$ due to its continuity. By condition the function $f(x)$ is stable on $\ell_\infty$, therefore,  function $F(x)$ differs from zero and we may put $F(x) = \sum_{k \in {{\mathbb Z}}} f(x \ominus k) = 1$. As function $f(x)$ is continuous and compactly supported, then for each $ \varepsilon > 0 $
    $$
    |f(x) - \sum_{k \in {{\mathbb Z}}} f(2^{-n}k) f(2^nx \ominus k)| < \varepsilon \ \ \forall x \in {{{\mathbb R}}}_+
    $$
    for large enough $n$. Futhermore,
    $$
    |\sum_{k \in {{\mathbb Z}}} \{f(2^{-n}k) - (S^n(\delta))(k)\} f(2^n x \ominus k)| < \varepsilon \ \ \forall x \in {{{\mathbb R}}}_+,
    $$
    and from the condition of stability of the function $f(x)$ it follows, that
    $$
    |f(2^{-n}k) - (S^n(\delta))(k)| < const \cdot \varepsilon \ \ \forall k \in {{\mathbb Z}}
    $$
    and for large enough $n$, that is, we have established the convergence of the subdivision scheme on $\delta$-sequence, and therefore on all sequences from ${\ell}_\infty$.
\end{proof}

Even in the classical case, studying the refinement equation for the existence of a continuous solution is not easy. It is known \cite{NPS}, that in the classical case when the necessary conditions are met (\ref{eq.nec}), the refinement equation always has a generalized compactly supported solution. Moreover, this solution is unique up to normalization. \cite[theorem 2.4.4]{NPS}. However, it may not be continuous. The criterion of continuity of the solution was first obtained in the work \cite{CH}, based on the developed in \cite{DL} matrix method. In work \cite{RF} This method was generalized to dyadic refinement equations. The idea of the method is as follows: instead of studying a functional equation, one can go to an equation for a vector function, that is, from a refinement equation to the so-called self-similarity equation, which will be described below. The self-similarity equation is a special case of the equation for fractal curves. We will prove this result in the most general form, for arbitrary fractal curves. The next chapter is fully devoted to the question of generating fractal curves on a dyadic half-line.

\subsection{Fractal curves on the dyadic half-line}

For an arbitrary pair of affine operators in $ \mathbb R^N$, consider a functional equation for a vector function $\bv: [0,1] \to \mathbb R^N$
\begin{equation}\label{fractal curves}
            \bv(t) = \left\{ \begin{array}{cc}
            A_0 \bv(2t), & t \in [0, \frac{1}{2}), \\
            A_1 \bv(2t \ominus 1), & t \in [\frac{1}{2}, 1].
                             \end{array}\right.
        \end{equation}
This equation with binary compression of an argument in literature is called {\em self-similarity equation}, by analogy with the same equation on the segment [0, 1] with the usual addition \cite{NPS}. In the classical case, the name is justified by the fact that many famous fractal curves, such as the de Rham curve, the Koch curve, etc. are solutions of such equations. Indeed, it follows from the equation that the arc of a curve $\bv(t)$ between $t = 0$ and $t = \frac12$ affine-like (by operator $A_0$) the whole curve, and so is (by operator $A_1$) the second arc of a curve, from $t = \frac12$ to $t =1$. So, the point $ \bv(\frac{1}{2}) $ divides the curve $\{ \bv(t), t \in [0, 1] \}$  into two arcs, each of which is affine-like throughout the curve. A fractal curve is a continuous solution of the equation (\ref{fractal curves}) (with the usual addition). On the dyadic half-line we mean the W-continuous solution.

We formulate the conditions under which two affine operators generate a fractal curve. For this we need an auxiliary
\begin{theorem}\label{Norm}
    For any two operators $A_0$, $A_1$ in $ \mathbb R^N $ and for each
    $\varepsilon > 0$ there exists a norm in $\mathbb R^N$ such, that
$$
     (\rho - \varepsilon) \|x\| \leq \max (\| {A}_0 x \|, \| {A}_1 x \|) \leq (\rho + \varepsilon) \|x\|
$$
for any $ x \in \mathbb R^N $, where $\rho$ is a joint spectral radius of operators, namely:
$$
    \rho \ = \ \lim_{r \to \infty} \max_{(d_1, ..., d_r) \in { \{0, 1 \} }^r } { \| A_{d_1} \cdot \ldots \cdot A_{d_r} \| }^{ \frac{1}{r} }.
$$
\end{theorem}

The proof of this theorem could be found in \cite{RS}.

Dyadic modulus of continuity of a function $f$ is a number $ \bw(f, \delta) $ such, that
$$
\forall \delta > 0 \ \ \bw(f, \delta) = \displaystyle{ {\sup}_{x, y \in [0, 2^{n-1}), 0 \leq y < \delta} } \{ \| f(x \oplus y) - f(x) \| \}
$$

If $ \bw(f, 2^{-n}) \leq const \cdot 2^{- \alpha n} $ for some $\alpha > 0$, then $ \bw(f, \delta) \leq const(f, \alpha) \cdot \delta^{\alpha}$

Dyadic Hölder exponent ${\alpha}_{f}$ of a function $f$ is:

$$
{\alpha}_{f} = \sup_{\alpha > 0} \{\alpha: \bw(f, \delta) \leq const(f, \alpha) \cdot \delta^\alpha \}.
$$

Let $A_0$, $A_1$ $\in Aff({{\mathbb R}}^N)$ и $\|A_0\| \leq \rho$, $\|A_1\| \leq \rho$, $\rho \in (0,1)$. Then

\begin{theorem}\label{Continuous fractal solution}
    If $\rho(A_0, A_1) <1$, then equation~(\ref{fractal curves}) has a continuous solution $\bv(t)$. Wherein,  ${\alpha}_{\bv} \geq - \log_2 \rho$. If there are no common affine subspaces of operators $ A_0, A_1 $, then ${\alpha}_{\bv} = - \log_2 \rho$.
\end{theorem}

\begin{proof}
    Consider binary rational $x$ and $y$ and without loss of generality we assume, that $x, y \in (0, 1)$ и $x < y$. Let us evaluate the norm $\|\bv(x) - \bv(y)\|$. Consider the set $P \subset {{\mathbb N}}$ of indices $p \in {{\mathbb N}}$ such, that $p$-th digits in the binary decomposition $x$ and $y$ are different. The set $P$ is finite due to $x$ and $y$ being binary rational. Let ${p}_0$ be the smallest element of $P$; it is obvious, that $|y \ominus x| \geq 2^{-{p}_0}$. Then
    $$
    \|\bv(y) - \bv(x)\| \ = \ \|\sum_{p \in {{\mathbb P}}} \bv(0.{d}_1 \ldots {d}_{p-1}1) - \bv(0.{d}_1 \ldots {d}_{p-1}0)\| \leq
    $$
    $$
    \leq \sum_{p \in {{\mathbb P}}} \| \bv(0.{d}_1 \ldots {d}_{p-1}1) - \bv(0.{d}_1 \ldots {d}_{p-1}0) \| \ =
    \sum_{p \in {{\mathbb P}}} \|{A}_{{d}_1} \ldots {A}_{{d}_{p-1}} ({A}_1 {\bv}_0 - {A}_0 {\bv}_0)\| \leq
    $$
    $$
    \leq \ \sum_{p \in {{\mathbb P}}} \| {A}_{{d}_1} \ldots {A}_{{d}_{p-1}} \| \cdot \| {A}_1 {\bv}_0 - {A}_0 {\bv}_0 \| \ \leq \
    \sum_{p \in {{\mathbb P}}} {( \rho({A}_0, {A}_1) + \varepsilon )}^{p-1} \cdot \| {A}_1 {\bv}_0 - {A}_0 {\bv}_0 \|,
    $$
    where ${\rho({A}_0, {A}_1)}$ is a joint spectral radius of operators ${A}_0$ и ${A}_1$. Let $ \sigma $ be a sum of a finite number series $\sum_{p \in {{\mathbb P}}} {( \rho({A}_0, {A}_1) + \varepsilon )}^{p-1}$. Then there is a constant ${C}_1$ such, that $ {C}_1 = \frac{\sigma}{ {( \rho({A}_0, {A}_1) + \varepsilon )}^{ {p}_0 - 1} } $. Let us denote $ {C}_2 = \| {A}_1 {\bv}_0 - {A}_0 {\bv}_0 \| = \| {A}_1 {\bv}_0 - {\bv}_0 \| $. Let $ C = {C}_1 \cdot {C}_2 $. We have:
    $$
    \sum_{p \in {{\mathbb P}}} {( \rho({A}_0, {A}_1) + \varepsilon )}^{p-1} \cdot \| {A}_1 {\bv}_0 - {A}_0 {\bv}_0 \| = C \cdot \frac{{( \rho({A}_0, {A}_1) + \varepsilon )}^{ {p}_0 }}{\rho({A}_0, {A}_1) + \varepsilon }
    $$
    Consequently,
    $$
    \|\bv(y) - \bv(x)\| \leq \tilde{C} {|y \ominus x|}^{-\log_2{\rho(A_0, A_1) + \varepsilon }},
    $$
    where $ \tilde{C} = \frac{C}{\rho({A}_0, {A}_1) + \varepsilon} $ and function $\bv(t)$ is uniformly continuous at binary-rational points of the interval (0,1), and, therefore, on the segment [0,1]. From the evaluation above it also follows, that
    $$
    {\alpha}_{\bv} \geq -\log_2{\rho({A}_0, {A}_1)}.
    $$

    If for the operators $ A_0, \ A_1 $ there are no common affine subspaces, then the reverse inequality is proved in a similar way (the proof almost literally coincides with similar arguments in the proof of theorem 5.1.4 of \cite{NPS}). Due to the fact that $ \varepsilon $ can be arbitrarily small, we obtain
    $$
    {\alpha}_{\bv} = -\log_2 \rho(A_0, A_1).
    $$

\end{proof}

Note that in the classical case the Theorem~\ref{Continuous fractal solution} is false.
For the existence of a continuous solution on the classical real line, an additional condition is necessary, sometimes called {\em cross condition} or {\em Barnsley condition}:   $A_0 {\bv}_1 = A_1 {\bv}_0$, where ${\bv}_0, {\bv}_1$ are the fixed points of operators $A_0, A_1$ from (\ref{fractal curves}) respectively. This condition eliminates the possibility of the function $\bv(t)$ to be discontinuous at the point $t = \frac{1}{2}$. In the dyadic case, $W$-continuity admits the existence of discontinuities at binary rational points, so that in the dyadic case this condition is not required.

\subsection{Spectral properties of subdivision schemes}

Now we apply the results of Section 3 on fractal curves to study the convergence of refinement schemes. Operator $S$ possesses two $ N \times N $ matrices, where $ N = 2^{n-1}, $ \\
$$
{T_0}_{ij} = c_{2(i-1)\oplus(j-1)}\, , \qquad
{T_1}_{ij} = c_{(2i-1)\oplus(j-1)}\, , \qquad  1 \leq i, j \leq N = 2^{n-1}
$$

We represent the explicit form of these matrices for $n = 3$:
$$
T_0 \ = \ 
\begin{pmatrix}
  c_0 & c_1 & c_2 & c_3 \\
  c_2 & c_3 & c_0 & c_1 \\
  c_4 & c_5 & c_6 & c_7 \\
  c_6 & c_7 & c_4 & c_5
\end{pmatrix}
$$

$$
T_1 \ = \
\begin{pmatrix}
  c_1 & c_0 & c_3 & c_2 \\
  c_3 & c_2 & c_1 & c_0 \\
  c_5 & c_4 & c_7 & c_6 \\
  c_7 & c_6 & c_5 & c_4
\end{pmatrix}
$$

By conditions (\ref{eq.nec}) for ${T}_0$, ${T}_1$ there is a common invariant affine subspace $W = \{ x \in {{{\mathbb R}}}^N: \sum_{i=1}^{N} x_i = 1 \}$. On this subspace of the matrices $T_0, \, T_1$ define affine operators.
We denote $ V \subset W $ as a smallest common affine subspace containing eigenvector $T_0$, corresponding to eigenvalue 1, and ${T}_0{|}_V = {A}_0$, \ ${T}_1{|}_V = {A}_1$.

Now we return to the refinement equation and apply the method described in \cite{DL}. So, for the refinement equation $\varphi(t) = \sum_{k=0}^N c_k \varphi(2t \ominus k)$ we define a vector-function $ \bv(t) = ( \varphi(t), \  \varphi(t+1), \  ..., \ \varphi(t-N+1) ) \in \mathbb R^N$ at arbitrary $ t $, besides,

\begin{equation}
            \bv(t) = \left\{ \begin{array}{cc}
            T_0 \bv(2t), & t \in [0, \frac{1}{2}], \\
            T_1 \bv(2t \ominus 1), & t \in [\frac{1}{2}, 1].
                             \end{array}\right.
\end{equation}

Since $W$ is invariant with respect to operators $ T_0, T_1 $, the function $ \bv $ also satisfies

\begin{equation}
            \bv(t) = \left\{ \begin{array}{cc}
            A_0 \bv(2t), & t \in [0, \frac{1}{2}], \\
            A_1 \bv(2t \ominus 1), & t \in [\frac{1}{2}, 1],
                             \end{array}\right.
\end{equation}

which combined with the Theorem \ref{Continuous fractal solution} leads to the following result:

\begin{theorem}\label{Continuous refinement equation solution}
    The refinement equation $\varphi(t) = \sum_{k=0}^{N} {c}_k \varphi(2t - k)$ possesses a continuous solution if and only if $\rho(A_0,A_1) < 1$.
    Moreover, $\alpha_{\varphi} = - \log_2 \rho(A_0, A_1)$.
\end{theorem}

\begin{corollary}
    If $\rho(T_0|_{W},T_1|_{W} ) < 1$, then the solution of the refinement equation is continuous and
    $\alpha_{\varphi} \ge - \log_2 \rho (T_0|_{W},T_1|_{W} )$.
\end{corollary}

\begin{remark}
    The corollary 1 has been proven also in\cite{RF}.
\end{remark}

Theorem \ref{Continuous refinement equation solution} is the consequence of Theorem \ref{Continuous fractal solution} and it is a criterion that the refinement equation has a continuous solution in terms of operators ${A}_0$, ${A}_1$. It allows us to establish whether the limit function of the subdivision scheme is continuous without finding its explicit form. And, accordingly, if the limit function is discontinuous, then there is no point in speaking about the convergence of the scheme.

So we first switched from the functional equation (\ref{refinement equation}) to the equation for the vector function(\ref{fractal curves}), figured out under what conditions it has a continuous solution (Theorem \ref{Continuous fractal solution}), then went back to the refinement equation (Theorem \ref{Continuous refinement equation solution}). Now it is easy to verify whether the solution is stable and, applying the Theorem \ref{Sufficient convergence condition}, find out if the corresponding subdivision scheme converges. We ended up with a convergence theorem for subdivision schemes, which is fairly easy to use.

Next, we consider some special cases of subdivision schemes.

\subsection{Dyadic subdivision schemes with positive coefficients}

We start with the subdivision schemes defined by positive coefficients. To begin with we will study some of their combinatorial properties. A matrix is called {\em stochastic in columns (rows)} if it is non-negative and the sum of all elements of each column (row) equals one. The following auxiliary result is well known in the theory of Markov chains. We give his proof for the convenience of the reader.

Recall that $ W \ = \ \{ x: \displaystyle{\sum_{k \in \mathbb Z} x_k = 1} \} $, $ \Delta \ = \ \{ x \in W: x \geq 0 \} $ is a unit simplex.

\begin{theorem}\label{Dyadic Stochastic}
    If $A$ is column-stochastic and possesses at least one positive row, and $x$, $y$ are the elements of $ \Delta $, then there exists $ q \in (0,1) $ such, that
    $$
    {\|Ax - Ay\|}_{{\ell}_1} \leq q{\|x - y\|}_{{\ell}_1}.
    $$
\end{theorem}

\begin{proof}
    Consider the function $\|x - y\|_{\ell_1} = f(x,y)$: on a compact set $ \Delta $ its maximum is reached at extreme points (because $ f(x, y) $ is a convex function), that is, points that are not the midpoints of any segments in  this set. Consequently, to prove that the distance between two points decreases under the action of a certain matrix, it is sufficient to prove that it decreases for the extreme points. Without loss of generality, let $x$ be the first basis vector in $W$ (with a 1 in the first place and zeros on all others), $y$ be the second one (with a 1 in second place and zeros on all others), and $\|x - y\|_{\ell_1} = 2$. Multiplying the matrix $ A $ by each of them, we get the first and second columns of it, respectively.
    It is necessary to prove that $\|Ax - Ay\|_{\ell_1} < 2$. Let $i$ be a number of the positive row of $A$, and $m < 1$ is the smallest element in the row. Element $Ax$ has at least $m$ on the $i$-th place. So does $Ay$. Consider $Ax$ and represent it as
    $$
    Ax \ = \ {x}_i + \sum_{j \neq i} {x}_j, \ {x}_i \geq m,
    $$
    where elements $ x_j $ have a non-zero component only on $ j $-th place.
    Similarly for $Ay$. For each $j \neq i$ it holds that
    $$
    |({Ax})_j - ({Ay})_j| \leq {x}_j - {y}_j.
    $$
    In the $i$-th row:
    $$
    |({Ax})_i - ({Ay})_i| \leq {x}_i - {y}_i - 2m.
    $$
    Consequently,
    $$
    {\| Ax - Ay \|}_{\ell_1} \leq \sum_{i \in {{\mathbb N}}} |{x}_i - {y}_i| - 2m \ = \ 2 - 2m,
    $$
    that is, the norm has changed $ 1-m $ times, $1 > 1-m > 0$.
\end{proof}

This fact allows us to prove the convergence of dyadic subdivision schemes with a positive mask with the help of the following property of convergence of schemes, the proof of which the reader can find in \cite{DynL} (in the case of the dyadic half-line, it completely repeats the classical analogue).

\begin{theorem}\label{Convergence property}
    If there is a norm in which the operators $T_0, \ T_1$ are contractions, then the subdivision scheme converges.
\end{theorem}

\begin{theorem}
    The subdivision scheme with a strictly positive mask converges.
\end{theorem}

\begin{proof}
    It is enough to establish that the matrices ${T}_0$, ${T}_1$ always have at least one positive row. Then by the Theorem \ref{Dyadic Stochastic} they are contractions in $ \ell_1 $ norm, which implies the convergence of the scheme according to the Theorem \ref{Convergence property}.
    
   Note that both ${T}_0$, ${T}_1$ are column stochastic. Let $I \ = \ \{i: {c}_i > 0\}$ be the set of indices, corresponding to positive mask coefficients of a particular subdivision scheme. Let also $N = 2^{n-1}$ be the dimension of matrices ${T}_0$, ${T}_1$, and $2^n$ is the number of coefficients. In order for each of the matrices to have at least one positive line, it is necessary that for a fixed $ i $ the following inequalities are satisfied:
   
$$
0 \leq 2(i-1) \oplus (j-1) \leq 2^n - 1,
$$
$$
0 \leq (2i-1) \oplus (j-1) \leq 2^n - 1.
$$
    
    That is, all indices of the elements of each matrix in some row belong to the set $I$. Let $i = 1$, consider ${T}_0$. It is remained to investigate whether the minimum and maximum of $i$-th row of ${T}_0$ are in $I$ or not. The inequality above will take the form of
    $$
    0 \leq j - 1 \leq 2^n - 1.
    $$
    The left part of the inequality is always satisfied. $j - 1$ will be maximum when $j = N$ ($j$ takes the highest possible even value). Therefore, we have
    $$
    2^{n-1} - 1 \leq 2^n - 1,
    $$
    which obviously also holds. We obtain that when $i = 1$ all the coefficients of the matrix ${T}_0$ are positive, which means the matrix has a positive row. \\
    Consider now ${T}_1$ and likewise we set $i = 1$. We have
    $$
    0 \leq j \oplus 2 \leq 2^n - 1.
    $$
    The left side of the inequality, again, is always fulfilled. Let the number $N = 2^{n-1}$ contain in its binary decomposition the rank of two. $j - 1$ will be maximum when $j = N \oplus 3$ ($j$ takes the highest possible even value). Then we have
    $$
    2^{n-1} - 1 \leq 2^n - 1,
    $$
    which is obviously fulfilled. If the number $N = 2^{n-1}$ do not contain in its binary decomposition the rank of two, then $j - 1$ will be maximum when $j = N$ ($j$ takes the maximum possible value). Then the inequality takes the following form:
    $$
    2^{n-1} \oplus 2 \leq 2^n - 1 = 2 \cdot 2^{n-1} - 1,
    $$
   which is true under these conditions. Therefore, the matrix $T_1$ also has a positive row and the theorem is proved.
\end{proof}

\subsection{Non-negative dyadic subdivision schemes}

We now consider subdivision schemes with a mask of non-negative elements. Let $I \ = \ \{i: {c}_i > 0\}$, ${I}_N \ = \ I \oplus 2I \oplus \ldots \oplus 2^{N-1}I$.  We present a criterion for the convergence of dyadic refinement schemes with nonnegative coefficients (A similar result for classical refinement schemes can be found in \cite{M}, its proof is completely transferred to the dyadic case):

\begin{theorem}\label{Convergence of the non-negative SS}
    Subdivision scheme with non-negative coefficients converges if and only if there exists $N$ such, that for each $i$ there is $j$ such, that $i \oplus l \ominus 2^N j$ in ${I}_N$, l = 0, \ldots, n - 1, $N = 2^{n-1}$
\end{theorem}

That is, set ${I}_N$ should be dense enough on the segment $ [\min \{i: i \in {I}_N \}, \max \{i: i \in {I}_N \} ]. $

Based on the criterion, we hypothesize:

\begin{conj}
    Subdivision scheme with nonnegative coefficients converges unless greatest common divisor of $ I \neq 1 $.
\end{conj}

\begin{remark}
    Note that the classical analogue of the hypothesis above includes another case where the schemes obviously do not converge, namely: $I$ has exactly one odd element and it is the last one (or $I$ has exactly one even element and it is the first one). If this condition is satisfied, then the solution of the refinement equation will be discontinuous, at least at zero, but in the case of dyadic subdivision schemes, such discontinuities are quite acceptable. Moreover, many numerical examples (section 8, example 9) show that when executed (\ref{Necessary convergence condition}), such dyadic subdivision schemes converge.
\end{remark}

We give an example of a scheme, which GCD($I$) $ \neq $ 1, and show that it diverges. Let the set of indices of non-negative coefficients be as follows: $I = \{0, 6, 9, 15\}$. Consider the sum
$$
    I \oplus 2I \ = \ \{0, 6, 9, 15; 12, 10, 5, 3; 18, 20, 27, 29; 30, 24, 23, 17\} \ = \ \{I, {I}_1, {I}_2, {I}_3\}
$$
It can be seen that there are no numbers $\{1, 2, 4, 7, 8, 11, 13, 14, 16, 19, 21, 22, 25, 26, 28\}$. That is, by Theorem \ref{Convergence of the non-negative SS} such a scheme does not converge. Indeed, the set $I \oplus 2I \oplus 4I$ does not fill these gaps, but only add new ones, for example:
$$
    I \oplus 2I \oplus 24 \ = \ I \oplus 2I.
$$
A whole series of numbers falls out, and this is repeated for the set ${I}_n$ at least once for each $n$. That is, the conditions of the criterion are obviously not fulfilled.

Hypothesis 1 states that the convergence of non-negative subdivision schemes can be simply verified. For arbitrary schemes, this is not so: the convergence of a subdivision scheme depends on the magnitude of the joint spectral radius of the matrices. $A_0, \, A_1$ (section 3). Another class of schemes whose convergence is relatively easy to verify is schemes with a small number of coefficients. In the next section, we show that the convergence of schemes with four coefficients is verified elementary. Note that in the classical theory of subdivision schemes this is not the case!

\subsection{Convergence of dyadic subdivision schemes with four coefficients}

Consider arbitrary dyadic subdivision scheme with coefficients $ \{ {c}_0, {c}_1, {c}_2, {c}_3 \} $. If it converges, then by the Theorem \ref{Necessary convergence condition} the sequence of its coefficients can be rewritten as
$ \{ c_0, c_1, 1 - c_0, 1 - c_1 \} $. Matrices $ T_0, T_1 $ have the form, respectively:

$$
\begin{pmatrix}
  c_0 & c_1 \\
  1 - c_0 & 1 - c_1
\end{pmatrix}
$$

$$
\begin{pmatrix}
  c_1 & c_0 \\
  1 - c_1 & 1 - c_0
\end{pmatrix}
$$

Consider the restriction of these matrices to a common linear invariant subspace $X \ = \ \{ {x} : {x}_1 + {x}_2 = 0 \} $ and represent them in basis $ {(1; -1)}^T $
We obtain:
$$
T_0 \cdot {(1; -1)}^T = (c_0 - c_1) \cdot {(1; -1)}^T
$$

$$
T_1 \cdot {(1; -1)}^T = (c_1 - c_0) \cdot {(1; -1)}^T
$$

It is known \cite{CDM}, that the subdivision scheme converges, if the joint spectral radius of $T_0, T_1$, which here equals to $ \max \{ | c_0 - c_1 | , | c_1 - c_0 | \}$, is less than one. We obtain: $ \max \{ | c_0 - c_1 | , | c_1 - c_0 | \} < 1$ and, therefore,

\begin{equation}\label{System of c_0, c_1}
 \begin{cases}
    c_0 - c_1 < 1, \\
    c_1 - c_0 < 1.
 \end{cases}
\end{equation}

The system above can be reduced to a double inequality: $ c_0 - 1 < c_1 < c_0 + 1 $. We depict its solution in the figure below:

\begin{figure}[h]
\begin{minipage}[h]{1\linewidth}
\center{\includegraphics[width=0.5\linewidth]{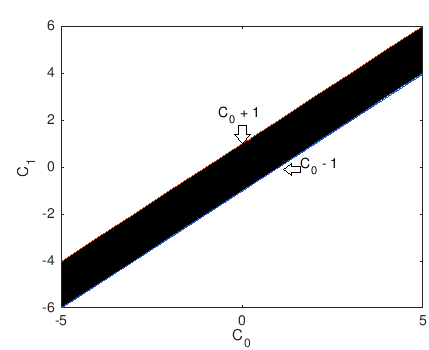} \\}
\end{minipage}
\hfill
\caption{The connection between $c_0$ and $c_1$.}
\label{ris:image01}
\end{figure}

Thus, all the schemes of the four coefficients, of which the first two satisfy the strict inequality above, obviously converge.

\begin{remark}
    The classical theory of subdivision schemes does not provide a complete description of all cases in which the scheme of four coefficients converges. (e.g. \cite{NPS, P_3}).
\end{remark}

\subsection{Numerical examples}

We present some illustrations of dyadic subdivision schemes and compare them with their classical counterparts.
For simplicity, we represent only the refinable functions of the subdivision schemes (the remaining sequences are obtained by various shifts, classical and dyadic, of these functions). All the images below were taken after ten iterations of the subdivision schemes. Example 6 illustrates the case when in the classical case the subdivision scheme diverges, and in the dyadic it converges.

\begin{enumerate}
    \item Let the scheme be given by a sequence of coefficients \\
$c = \{ 0.3, 0.1, 0.7, 0.9 \}$. We construct the refinable function of this scheme in the classical and dyadic case.

\begin{figure}[h]
\begin{minipage}[h]{0.49\linewidth}
\center{\includegraphics[width=1\linewidth]{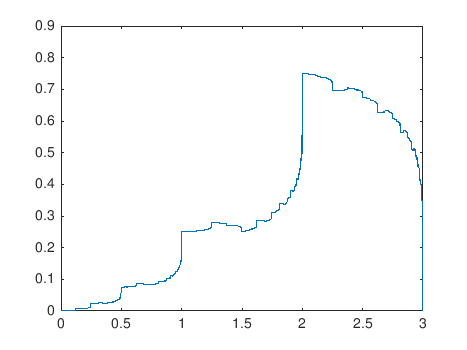} \\ 1)Classic case}
\end{minipage}
\hfill
\begin{minipage}[h]{0.49\linewidth}
\center{\includegraphics[width=1\linewidth]{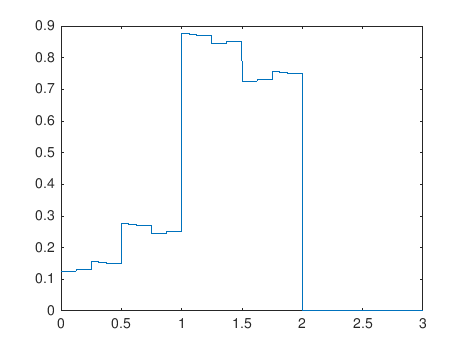} \\ 2)Dyadic case}
\end{minipage}
\caption{}
\label{ris:image1}
\end{figure}

In both cases the refinable function is continuous. In the classical case it is supported by a segment $ [0; 3] $. At the point $ x = 3 $ the function is zero, despite it is quite unclear on the figure 3 (the function rapidly decreases to zero). In the dyadic case the refinable function is supported by the segment $ [0; 2] $ and is zero at the point $ x = 2 $ correspondingly. Wherein at the point $ x = 2 $ there is a discontinuity, which, as we know, does not contradict the definition of $W$-continuity.

\item Next, we present a scheme with coefficients

$c = \{ 0.6, 0.9, 0.4, 0.1 \}$.

\begin{figure}[h]
\begin{minipage}[h]{0.49\linewidth}
\center{\includegraphics[width=1\linewidth]{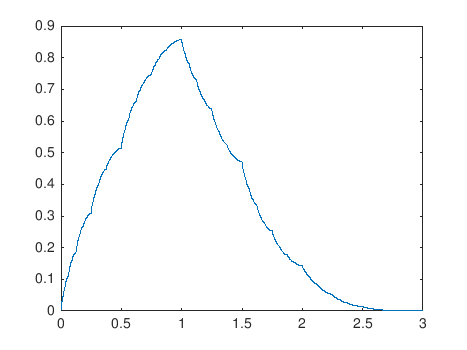} \\ 1)Classic case}
\end{minipage}
\hfill
\begin{minipage}[h]{0.49\linewidth}
\center{\includegraphics[width=1\linewidth]{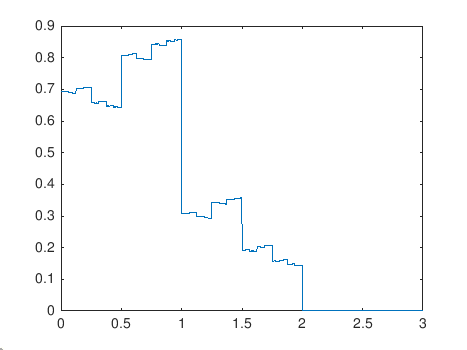} \\ 2)Dyadic case}
\end{minipage}
\caption{}
\label{ris:image2}
\end{figure}

\newpage

\item Now the coefficients are equal to $c = \{ 0.6, 1.1, 0.4, -0.1 \}$.

\begin{figure}[h]
\begin{minipage}[h]{0.49\linewidth}
\center{\includegraphics[width=1\linewidth]{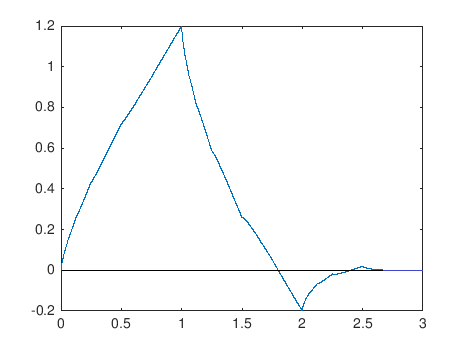} \\ 1)Classic case}
\end{minipage}
\hfill
\begin{minipage}[h]{0.49\linewidth}
\center{\includegraphics[width=1\linewidth]{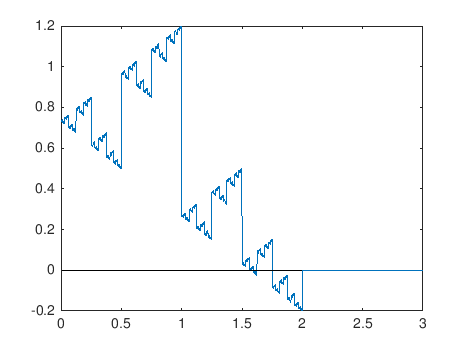} \\ 2)Dyadic case}
\end{minipage}
\caption{}
\label{ris:image3}
\end{figure}

\item Now $c = \{ \frac{1}{4}, \frac{3}{4}, \frac{3}{4}, \frac{1}{4} \}$.

\begin{figure}[h]
\begin{minipage}[h]{0.49\linewidth}
\center{\includegraphics[width=1\linewidth]{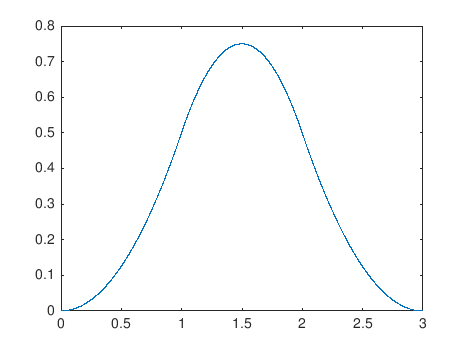} \\ 1)Classic case}
\end{minipage}
\hfill
\begin{minipage}[h]{0.49\linewidth}
\center{\includegraphics[width=1\linewidth]{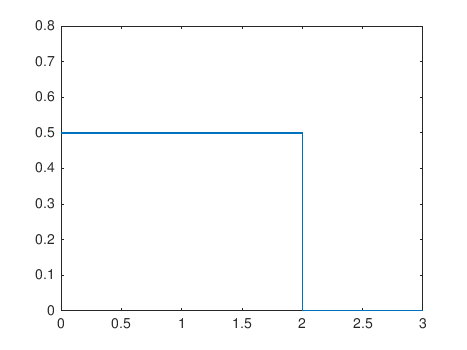} \\ 2)Dyadic case}
\end{minipage}
\caption{}
\label{ris:image4}
\end{figure}

\newpage

\item Let us give an example of divergent scheme: $c = \{ 2.6, 0.7, -1.6, 0.3 \}$.

\begin{figure}[h]
\begin{minipage}[h]{0.49\linewidth}
\center{\includegraphics[width=1\linewidth]{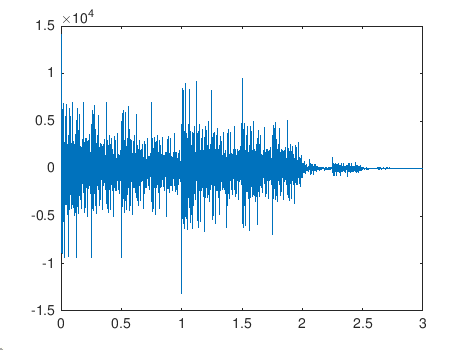} \\ 1)Classic case}
\end{minipage}
\hfill
\begin{minipage}[h]{0.49\linewidth}
\center{\includegraphics[width=1\linewidth]{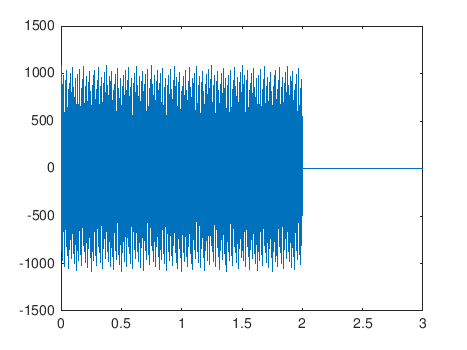} \\ 2)Dyadic case}
\end{minipage}
\caption{}
\label{ris:image5}
\end{figure}

\item Let us give an example, when in the classical case the subdivision scheme diverges, and in the dyadic case it converges:

$c = \{ 0.4, -0.1, 0.6, 1.1 \}$.

\begin{figure}[h]
\begin{minipage}[h]{0.49\linewidth}
\center{\includegraphics[width=1\linewidth]{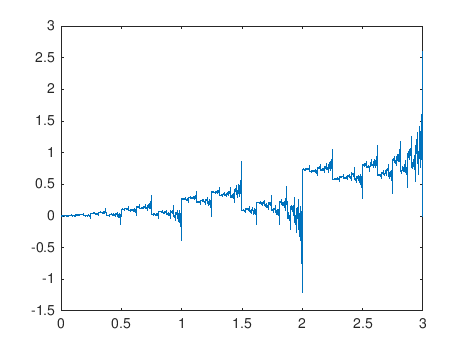} \\ 1)Classic case}
\end{minipage}
\hfill
\begin{minipage}[h]{0.49\linewidth}
\center{\includegraphics[width=1\linewidth]{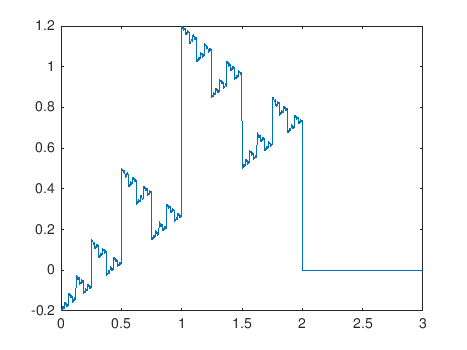} \\ 2)Dyadic case}
\end{minipage}
\caption{}
\label{ris:image6}
\end{figure}

In the classical case, it cannot converge, since the last coefficient is greater than one (see \cite{CDM}). In the dyadic case, it converges for reasons from the section 7: the conditions \ref{System of c_0, c_1} for the first two coefficients are fulfilled:

$$
    c_0 - 1 < c_1 < c_0 + 1,
$$
$$
    0.4 - 1 < -0.1 < 0.4 + 1,
$$

and, therefore, such a scheme converges.

\newpage

\item Consider the schemes given by the eight coefficients. Put all eight coefficients equal to $ \frac{1}{4} $:

\begin{figure}[h]
\begin{minipage}[h]{0.49\linewidth}
\center{\includegraphics[width=1\linewidth]{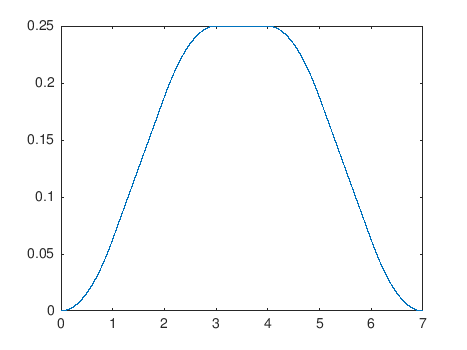} \\ 1)Classic case}
\end{minipage}
\hfill
\begin{minipage}[h]{0.49\linewidth}
\center{\includegraphics[width=1\linewidth]{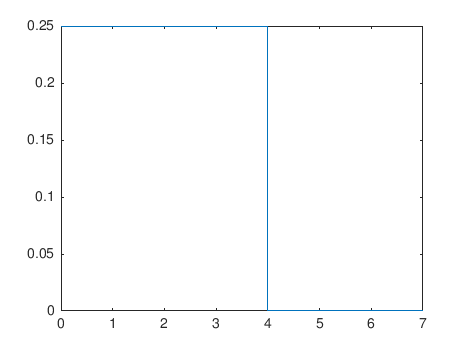} \\ 2)Dyadic case}
\end{minipage}
\caption{}
\label{ris:image7}
\end{figure}

\item Now $c = \{ \frac{1}{8}, \frac{3}{8}, \frac{1}{8}, \frac{3}{8}, \frac{3}{8}, \frac{1}{8}, \frac{3}{8}, \frac{1}{8} \}$.

\begin{figure}[h]
\begin{minipage}[h]{0.49\linewidth}
\center{\includegraphics[width=1\linewidth]{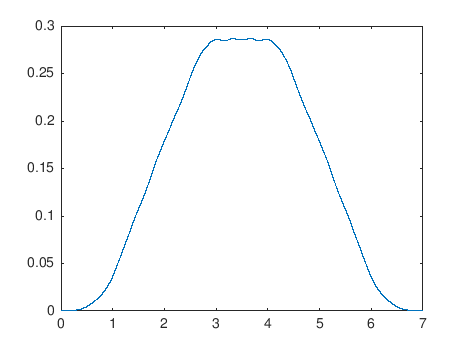} \\ 1)Classic case}
\end{minipage}
\hfill
\begin{minipage}[h]{0.49\linewidth}
\center{\includegraphics[width=1\linewidth]{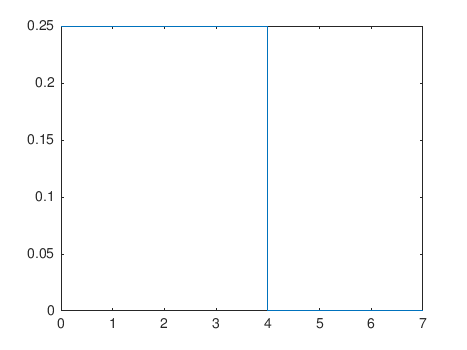} \\ 2)Dyadic case}
\end{minipage}
\caption{}
\label{ris:image8}
\end{figure}

\newpage

\item Let us give an example of a dyadic interpolating ( $ c_0 \ = \ 1 $, the rest of the even coefficients are zeros) subdivision scheme defined by sixteen coefficients:

$c = \{ 1, \frac{1}{8}, 0, \frac{1}{8}, 0, \frac{1}{8}, 0, \frac{1}{8}, 0, \frac{1}{8}, 0, \frac{1}{8}, 0, \frac{1}{8}, 0, \frac{1}{8} \}$.

\begin{figure}[h]
\begin{minipage}[h]{0.49\linewidth}
\center{\includegraphics[width=1\linewidth]{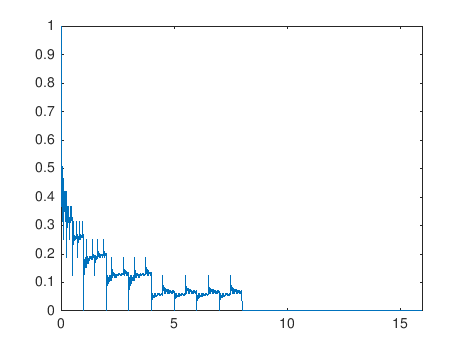} \\ 1)refinable function}
\end{minipage}
\hfill
\begin{minipage}[h]{0.49\linewidth}
\center{\includegraphics[width=1\linewidth]{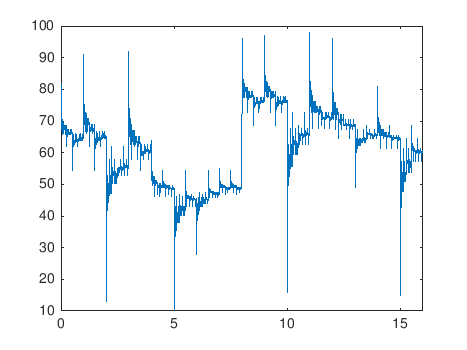} \\ 2)An arbitrary sequence of 16 elements in the range [1:100]}
\end{minipage}
\caption{}
\label{ris:image9}
\end{figure}

\end{enumerate}

The author is grateful to Professor V. Yu. Protasov for valuable advice and remarks.


\begin{thebibliography}{NN}

\bibitem{G}
Голубов Б.\,И.,
\newblock {\em Элементы двоичного анализа},
\newblock ЛКИ.2007.
\smallskip

\bibitem{NPS}
Новиков И.\,А., Протасов В.\,Ю., Скопина М.\,А.,
\newblock {\em Теория всплесков},
\newblock ФизМатЛит.2005.
\smallskip

\bibitem{P_1}
Протасов В.\,Ю,
\newblock {\em Аппроксимация диадическими всплесками},
\newblock Математический сборник, том 198, номер 11, 135–152. 2007.
\smallskip

\bibitem{P_3}
Протасов В.\,Ю,
\newblock {\em Спектральное разложение 2-блочных тёплицевых матриц и масштабирующие уравнения},
\newblock Алгебра и анализ, 18:4 (2006), 127–184; St. Petersburg Math. J., 18:4 (2007), 607–646.
\smallskip

\bibitem{P_4}
Протасов В.\,Ю.,
\newblock {\em Фрактальные кривые и всплески},
\newblock Изв. РАН. Сер. матем., 70:5 (2006), 123–162.
\smallskip

\bibitem{RF}
Родионов Е.\,А., Фарков Ю.\,А.,
\newblock {\em Оценки гладкости диадических ортогональных всплесков типа Добеши},
\newblock Матем. заметки, 86:3 (2009), 429–444; Math. Notes, 86:3 (2009), 407–421.
\smallskip

\bibitem{BDF}
Bendorya T., Dekelb S., Feuera A.,
\newblock {\em Robust recovery of stream of pulses using convex optimization},
\newblock J. Math.Anal.Appl.442 511–536. 2016.
\smallskip

\bibitem{CDM}
Cavaretta A.\,S., Dahmen W., Micchelli Ch.\,A.,
\newblock {\em Stationary subdivision},
\newblock Mem. Amer. Math. Soc., 93:453 (1991), 186.
\smallskip

\bibitem{Ch}
Chaikin G.\,M.,
\newblock {\em An algorithm for high speed curve generation},
\newblock Comput. Graphics and Image Processing, 3 (1974), 346–349.
\smallskip

\bibitem{ChP}
Charina M., Protasov V.\,Yu.,
\newblock {\em Smoothness of anisotropic wavelets and subdivisions},
\newblock 2017.
\smallskip

\bibitem{CH}
Colella D., Heil C.,
\newblock {\em Characterizations of refinable functions: continuous solutions},
\newblock SIAM J. Matrix Anal. Appl. 15 (1994), no. 2, 496–518.
\smallskip

\bibitem{D}
Daubechies I.,
\newblock {\em Ten lectures on wavelets},
\newblock SIAM.1992. 
\smallskip

\bibitem{DL}
Daubechies I., Lagarias J.,
\newblock {\em Two-scale difference equations},
\newblock I. Global regularity of solutions, SIAM. J. Math. Anal. 22 (1991), 1388–1410.
\smallskip

\bibitem{DeR}
De Rham G.,
\newblock {\em Sur les courbes limités de polygones obtenus par trisection},
\newblock Enseign. Math. (2), 5 (1959), 29-43.
\smallskip

\bibitem{DD}
Deslauriers G., Dubuc S.,
\newblock {\em Symmetric iterative interpolation processes},
\newblock Constr. Approx., 5:1 (1989), 49–68.
\smallskip

\bibitem{Dyn}
Dyn N.,
\newblock {\em Linear and Nonlinear Subdivision Schemes in Geometric Modeling},
\newblock Tel Aviv University.
\smallskip

\bibitem{DynL}
Dyn N., Levin D.,
\newblock {\em Subdivision Schemes in Geometric Modelling},
\newblock Tel Aviv University.
\smallskip

\bibitem{DynLG}
Dyn N., Levin D., Gregory J.\,A.,
\newblock {\em A Butterfly Subdivision Scheme for Surface Interpolation with Tension Control},
\newblock Tel Aviv University, Brunel University. ACM Transactions on Graphics, Vol. 9, No. 2, April 1990, Pages 160-169.
\smallskip

\bibitem{DynLM}
Dyn N., Levin D., Marinov M.,
\newblock {\em Geometrically Controlled 4-Point Interpolatory Schemes},
\newblock Springer-Verlag, 2004.
\smallskip

\bibitem{GES}
ГB. I. Golubov, A. V. Efimov, V. A. Skvortsov,
\newblock {\em  Walsh series and transforms: theory and
applications},
\newblock Nauka, Moscow 1987; English transl., Kluwer, Dordrecht 1991. 
\smallskip

\bibitem{M}
Melkman A.\,A.,
\newblock {\em Subdivision schemes with non-negative masks converge always - unless they obviously cannot},
\newblock Baltzer Journals. 1996.
\smallskip

\bibitem{PF}
Protasov V.\,Yu., Farkov Yu.\,A,
\newblock {\em Dyadic wavelets and refinable functions on a half-line},
\newblock Sbornic: Mathematics. 2006. Vol.~197, No.~10, pp.~129--160. 
\smallskip

\bibitem{RS}
Rota G.-C., Strang G.,
\newblock {\em A note on the joint spectral radius},
\newblock Indag. Math., 22 (1960).
\smallskip

\end{thebibliography}
\end{document}